\documentclass{article}
\usepackage[utf8]{inputenc}
\usepackage{latexsym}
\usepackage{amsthm}
\usepackage{amssymb}
\usepackage{amsmath}
\usepackage{mathrsfs}
\usepackage{amsfonts}
\usepackage{indentfirst}
\usepackage[a4paper, total={6in, 10in}]{geometry}
\newcommand{\newcontent}{\vspace{.08in}\noindent}

\usepackage{setspace}
\usepackage{multirow}
\usepackage{array}
\usepackage{booktabs}
\allowdisplaybreaks
\usepackage{graphicx}
\usepackage{caption}
\usepackage{subfigure}
\newtheorem{remark}{Remark}

\newtheorem{lemma}{Lemma}

\newtheorem{theorem}{Theorem}

\newtheorem{assumption}{Assumption}
\usepackage{amsmath, bm}
\usepackage{algorithm}
\usepackage{algorithmic}
 
\begin{document}
	
\title{\bf Optimal prediction for kernel-based semi-functional linear regression 
}
\author{Keli Guo${\rm ^{1}}$, Jun Fan${\rm ^{1}}$, Lixing Zhu${\rm ^{1,2}}$ \\~\\
	{\small {\small {\it ${\rm ^1}$ Department of Mathematics, Hong Kong Baptist University, Kowloon, Hong Kong} }}\\
	{\small {\small {\it ${\rm ^2}$ Center for Statistics and Data Science, Beijing Normal University, Zhuhai, China} }}
}
\date{}
\maketitle

\begin{abstract}
	
	\noindent In this paper, we establish minimax optimal rates of convergence for prediction in a  semi-functional linear model that consists of a functional component and a less smooth nonparametric component. Our results reveal that the smoother functional component can be learned with the minimax rate as if the nonparametric component were known. More specifically, a double-penalized least squares method is adopted to estimate both the functional and nonparametric components within the framework of reproducing kernel Hilbert spaces. By virtue of the representer theorem, an efficient algorithm that requires no iterations is proposed to solve the corresponding optimization problem, where the regularization parameters are selected by the generalized cross validation criterion. Numerical studies are provided to demonstrate the effectiveness of the method and to verify the theoretical analysis. 
	
	\newcontent
\end{abstract}

\section{Introduction}

 Learning with intrinsically infinite-dimensional data has been an active area of research over the past three decades, which refers to as functional data analysis originally introduced by Ramsay \cite{FDA1, FDA2} in the 1980s. Since then a considerable amount of work has been devoted to functional regression problems where responses or covariates include functional data, see \cite{FDA3, FDA4, FDA6, FDA5, FLR1, FLR2, FLR3} and references therein. Among them one of the most commonly studied framework is the functional linear regression model, where a linear relationship between a scalar response and a functional covariate is postulated. More specifically, classical functional linear model has the form as follows:
\[Y=\int_{\mathcal{T}} X(t)\beta^{0}(t)dt+\epsilon,\]
where $Y$ denotes a centered scalar response, and $X: \mathcal{T}\to\mathbb{R}$ takes values in $\mathcal{L}_2:=\mathcal{L}_2(\mathcal{T})$, which consists of square integrable functions defined on a compact subset $\mathcal{T}$ of $\mathbb{R}$. The slope function is represented by $\beta^0$, and $\epsilon$ is the mean zero random error that is independent of $X$. To overcome the difficulty brought by the infinite-dimensional nature of this setting, a variety of methods have been proposed in the literature for both estimation and prediction tasks. Broadly speaking, most of the existing methods can be categorized into two streams: functional principal component analysis (FPCA)-based methods \cite{FDA6, FLR4, FDA5} and reproducing kernel Hilbert space (RKHS)-based methods \cite{FLR2, FLR3}. As introduced in \cite{FLR4}, the conventional FPCA-based methods employ spectral expansions of the covariance kernel of $X$ and its empirical estimator to learn the slope function. Later, Hall and his collaborators \cite{FLR1, FDA5} established the minimax convergence rate of FPCA-based estimators, which leads to the popularity of this type of methods. However, it is often not realistic to assume that the slope function could be represented by the leading functional principal components of the covariance kernel, and FPCA-based methods may fail in many circumstances. In recent years, Cai and Yuan \cite{FLR2, FLR3} introduced the RKHS framework into functional linear regression in order to overcome the inherent problem of FPCA-based methods by assuming that the target slope function resides in an RKHS. Moreover,  under weaker conditions compared to \cite{FLR1, FDA5}, they showed that the minimax optimal convergence rates for estimation and prediction can be achieved by the proposed smoothness regularization method.

However, in practice, we may encounter datasets including functional and scalar covariates simultaneously. In this situation, the classical functional linear regression model may not be appropriate and sufficient, which motivates researchers to propose new models to deal with functional linear regression with mixed covariates. There are mainly three types of models depending on the parametric or nonparametric natures of functional and scalar covariates. To incorporate the additional scalar covariates, \cite{Partialfunctional1} established the semi-functional partial linear regression model, where a nonparametric model for the functional covariate and a linear model for the scalar covariates are included and considered. Later, \cite{Partialfunctional2, Partialfunctional3} extended this model to the cases of dependent data and unknown error density, respectively. Other than the semi-functional partial linear regression model, \cite{Partialfunctional4} proposed the partially functional linear model, which admits a linear relationship between the response variable and both functional and scalar covariates. Later, \cite{Partialfunctional5} extended this model to deal with the case where high-dimensional scalar covariates and multiple functional covariates are considered. Recently, \cite{Partialfunctional6} first adopted the RKHS framework in the partially functional linear model, but the model cannot handle high-dimensional scalar covariates. \cite{Partialfunctional7} further improved the partially functional linear model by imposing $\ell^1-$penalty on the high-dimensional scalar covariates in order to achieve sparsity. Apart from two aforementioned models, \cite{Partialfunctional8} considered the so-called semi-functional linear regression model that consists of the functional linear regression component and the nonparametric component for the scalar covarites. The estimation approach of \cite{Partialfunctional8} relies on the previously mentioned FPCA-based methods for functional component and the Nadaraya-Watson kernel estimators for the nonprametric component. Later, \cite{Partialfunctional9} focused on a special case of \cite{Partialfunctional8} where the scalar covariate is one-dimensional and spline functions are employed to estimate both the functional and nonparametric components. More recently, \cite{Partialfunctional10} proposed a robust semi-functional linear regression model that is resistant against outliers or heavy-tailed noises.

In this paper, we focus on the semi-functional linear model consisting of both functional and nonparametric components as follows:
\begin{equation}\label{formula1}
	\begin{aligned}
	Y=\int_{\mathcal{T}} X(t)\beta^{0}(t)dt+g^0(\bm{Z})+\epsilon,
	\end{aligned}
\end{equation}
where $X\in\mathcal{X}\subset \mathcal{L}_2$ is a functional covariate, $\bm{Z}\in \mathcal{Z}\subset\mathbb{R}^p$ is a $p$-dimensional vector covariate, and the error $\epsilon$ follows a $\mathcal{N}(0, \sigma^2)$ distribution and is independent of $X$ and $\bm{Z}$. Without loss of generality, we assume that the functional component is smoother than the nonparametric component and also $\mathbb{E}[X|\bm{Z}]=0$. By assuming that both the slope function $\beta^0$ and the nonparametric function $g^0$ reside in some reproducing kernel Hilbert spaces, the goal of this paper is to investigate semi-functional linear regression within the framework of learning theory. 

We summarize our main contributions of the work as follows.
\begin{itemize}
	\item We establish minimax optimal rates of convergence for prediction in the semi-functional linear model. We find that for each component, the convergence rate is not affected by the other component. In other words, the functional component can be estimated with the same convergence rate even if an accurate estimate of the nonparametric component is not available. In the same way, the convergence rate of the nonparametric component remains the same as if the true $\beta^0$ is known. In the literature of additive model, see \cite{Additivemodel}, this phenomenon has been observed and discussed .
	\item We adopt a double-penalized least squares method to estimate both the functional and nonparametric components within the framework of reproducing kernel Hilbert spaces. Moreover, we propose an efficient algorithm that requires no iterations to solve the corresponding problem, where the generalized cross validation criterion (GCV) is applied to select the regularization parameters. Numerical studies are provided to demonstrate the effectiveness of the proposed method and to verify the theoretical analysis.
\end{itemize} 

The rest of the paper is organized as follows. In Section \ref{notation}, we introduce the semi-functional linear regression model within the framework of reproducing kernel Hilbert spaces and present the corresponding computational algorithm. Section \ref{results} presents the theoretical properties of the proposed double-penalized least squares method and derives minimax optimal convergence rates for prediction error of both functional and nonparametric components. Numerical studies are conducted in Section \ref{simulation} to demonstrate the effectiveness of the method and to verify the theoretical analysis. In Section \ref{conclusion}, we provide a summary of this paper. The proofs of our main results and some key lemmas are given in Appendix.

\section{Preliminaries and Methodology} \label{notation}
For any vector $\bm{v}\in \mathbb{R}^n$, we define the $\ell_1$ norm, Euclidean norm and $\ell_\infty$ norm as $\lVert \bm{v}\rVert_1=\sum_{i=1}^{n}\lvert v_i\rvert$, $\lVert \bm{v}\rVert_2=\sqrt{\sum_{i=1}^{n} v^2_i}$ and $\lVert \bm{v}\rVert_\infty=\max_{1\leq i\leq n}\lvert v_i\rvert$ respectively. The empirical norm is  written as $\lVert \bm{v}\rVert^2_n=\bm{v}^T\bm{v}/n$. For any measurable function $f: \mathcal{X} \times \mathcal{Z} \to\mathbb{R}$, we define function norms $\lVert f\rVert=\sqrt{\mathbb{E}f^2(X,\bm{Z})}$, $\lVert f\rVert_\infty=\sup_{X,\bm{Z}}\lvert f(X,\bm{Z})\rvert$ and $\lVert f\rVert^2_n=\frac{1}{n}\sum_{i=1}^{n} f^2(X_i,\bm{Z}_i)$. Moreover, for any functions $\beta_1, \beta_2\in \mathcal{L}_2$, we define $\lVert \beta_1\rVert^2_{\mathcal{L}_2}=\langle \beta_1,\beta_1\rangle_{\mathcal{L}_2}$ and $\langle \beta_1,\beta_2\rangle_{\mathcal{L}_2}=\int_{\mathcal{T}}\beta_1(t)\beta_2(t)dt$. Let $\{a_k: k\geq 1\}$ and $\{b_k: k\geq 1\}$ be two sequences, we write $a_k \asymp b_k$ if $C_1\leq \lim \inf a_k/b_k\leq \lim \sup a_k/b_k\leq C_2$ for some constants $C_1,C_2$. Furthermore, we let $a_k=\mathcal{O}(b_k)$ denote the case where $\lvert a_k\rvert\leq C_3\lvert b_k\rvert$ for some constant $C_3$. 
\subsection{Reproducing kernel Hilbert spaces}
We begin with a brief review of several basic facts about RKHS, which can be found in, for example, \cite{Representer}. We denote $\mathcal{T} \subset \mathbb{R}$ as a compact set. Recall that a reproducing kernel $K$: $\mathcal{T} \times \mathcal{T} \to \mathbb{R}$ is continuous, square integrable, symmetric and nonnegative definite. One reporducing kernel $K$ corresponds to an RKHS $\mathcal{H}_K$ which is a linear functional space equipped with the inner product $\langle \cdot,\cdot \rangle_{K}$. The RKHS $\mathcal{H}_K$ has the following reproducing property:
\[f(t)=\langle K(t,\cdot),  f \rangle_{K},\] 
for any $t\in\mathcal{T}$ and $f\in \mathcal{H}_K$. We denote a centered square integrable stochastic process defined over $\mathcal{T}$ by $X(\cdot)$. The covariance kernel of $X$ can be defined as  
\[C(s,t)=\mathbb{E} [X(s)X(t)], \;\;\;\; \forall t,s \in \mathcal{T},\]
which is real, symmetric and nonnegative definite. 

For any real, symmetric and nonnegative definite function $U:\mathcal{T}\times \mathcal{T}\to \mathbb{R}$, an integral operator $L_U: \mathcal{L}_2\to \mathcal{L}_2$ can be defined as 
\[L_U(f)(\cdot)=\langle U(s,\cdot),f\rangle_{\mathcal{L}_2}=\int_{\mathcal{T}}U(s,\cdot)f(s)ds.\]
By the spectral theorem, we have that $U$ admits the decomposition as follows
\[U(s,t)=\sum_{k\geq1}\theta^U_k\psi^U_k(s)\psi^U_k(t),  \;\;\;\; \forall t,s \in \mathcal{T},\]
where $\theta^U_1\geq \theta^U_2\geq ...\geq0$ are its eigenvalues and $\{\psi^U_k:k\geq1\}$ are the corresponding orthonormalized eigenfunctions. Moreover, we have $L_U(\psi^U_k)=\theta^U_k\psi^U_k$ for $k\geq1.$

Next, we define a linear operator $L_{U^{1/2}}$ as $L_{U^{1/2}}(\psi^U_k)=\sqrt{\theta^U_k}\psi^U_k,$
where 
\[U^{1/2}(s,t)=\sum_{k\geq1}\sqrt{\theta^U_k}\psi^U_k(s)\psi^U_k(t), \;\;\;\; \forall t,s \in \mathcal{T}.\]
We can observe that $L_{U^{1/2}}=L^{1/2}_U$. Define 
$(U_1U_2)(s,t)=\int_{\mathcal{T}}U_1(s,u)U_2(u,t)du,$
then we have $L_{U_1U_2}=L_{U_1}\circ L_{U_2}$. 
Given the reproducing kernel $K$ and covariance kernel $C$, we can define the linear operator $\mathscr{L}:=L_{{K}^{1/2}C{K}^{1/2}}=L_{K^{1/2}}(L_C(L_{K^{1/2}}(f)))$. Then we have for $\mathscr{L}$ that
\[\mathscr{L}\varphi_k=s_k\varphi_k, \;\;\; k\geq1,\]
where $s_1\ge s_2\ge\cdots\geq0$ are its eigenvalues and $\{\varphi_k: k\ge1\}$ are the corresponding orthonormalized eigenfunctions.
We can observe that the eigenvalues of $K$ and the eigenvalues of $C$ jointly determine the eigenvalues $\{s_k: k\ge1\}$ of the linear operator $\mathscr{L}$.

Similarly, given another reproducing kernel $G$: $\mathcal{Z} \times \mathcal{Z}\to \mathbb{R}$ that is continuous, square integrable, symmetric and nonnegative definite, we define its corresponding RKHS as $\mathcal{H}_G$ which is a linear functional space equipped with the inner product $\langle \cdot,\cdot \rangle_{G}$. Denote by $\rho_z$ the marginal distribution of $\bm{Z}$ and by $\mathcal{L}^2_{\rho_z}$ be the space of square-integrable functions on $\mathcal{Z}$ with norm $\|\cdot\|_{\mathcal{L}^2_{\rho_z}}$. We define the integral operator $L_G$ by
\[L_{G}(g)(\cdot)=\langle G(\bm{Z},\cdot),g\rangle_{\mathcal{L}^2_{\rho_z}}=\int_{\mathcal{Z}}G(\bm{Z},\cdot)g(\bm{Z})d\rho_{z}(\bm{Z}).\]
According to the spectral theorem, we have that $G$ admits the decomposition as follows 
\[G(\bm{Z}, \bm{Z'})=\sum_{k\geq1}\tilde{s}_k\tilde{\varphi}_k(\bm{Z})\tilde{\varphi}_k(\bm{Z'}), \;\;\;\; \forall \bm{Z},\bm{Z'} \in \mathcal{Z},\]
where $\tilde{s}_1\ge \tilde{s}_2\ge\cdots>0$ are its eigenvalues and $\{\tilde{\varphi}_k: k\ge1\}$ are the corresponding orthonormalized eigenfunctions. Moreover, we have $L_G\tilde{\varphi}_k=\tilde{s}_k\tilde{\varphi}_k$, for  $k\geq1$.
\subsection{Complexity of RKHS} \label{Rademacher}
To bound empirical processes, we are also interested in some data-dependent estimates of
the complexity of RKHS, namely, Rademacher and Gaussian complexities, see \cite{FLR3} for more details.
\begin{lemma} \label{lemmatech1} Suppose the eigenvalues $\{s_k:k\ge1\}$ of the linear operator $\mathscr{L}$ satisfy $s_k\leq c_1k^{-2r}$ for some constant $0<r<\infty$.
	Consider the following Rademacher type of process:
	\[W_n(b)=\frac1n\sum_{i=1}^{n}\omega_i\langle X_i, b\rangle_{\mathcal{L}_2},\]
	where $\omega'_i$s are i.i.d. Rademacher random variables, i.e., $\mathbb{P}(\omega_i=1)=\mathbb{P}(\omega_i=-1)=\frac12$. Define
	\[\|W_n\|_{\mathcal{B}(\rho, \kappa)}=\sup_{b\in\mathcal{B}(\rho, \kappa)}|W_n(b)|,\]
	where $\mathcal{B}(\rho, \kappa)=\{b\in\mathcal{H}_K: \|b\|_{K}\le \rho \hbox{ and } \|L_C^{\frac12}b\|_{\mathcal{L}_2}\leq \kappa\}$. Then we have \[\mathbb{E}\|W_n\|_{\mathcal{B}(\rho, \kappa)}\leq \frac1{\sqrt{n}}2c^{\frac{1}{4r}}_1\rho^{\frac1{2r}}\kappa^{1-\frac1{2r}}.\]
	Note that this lemma also holds for Gaussian type of process:
		\[G_n(b)=\frac1n\sum_{i=1}^{n}\epsilon_i\langle X_i, b\rangle_{\mathcal{L}_2},\]
	where $\epsilon'_i$s are i.i.d. standard Gaussian variables.
\end{lemma}
\begin{proof}
	It is clear that $\mathcal{B}(\rho, \kappa)=L_{K}^{\frac12}(\tilde{\mathcal{F}}_1(\rho, \kappa))$, where
	\[\tilde{\mathcal{F}}_1(\rho, \kappa)=\{\tilde{f}\in \mathcal{L}_2: \|\tilde{f}\|_{\mathcal{L}_2}\le \rho \hbox{ and } \|\mathscr{L}^{\frac12}\tilde{f}\|_{\mathcal{L}_2}\leq \kappa\}.\]
	Denote $\tilde{\mathcal{B}}=\left\{\sum_{k\ge1}\tilde{b}_k\varphi_k: \sum_{k\ge1}\left(\frac{\tilde{b}_k}{\min\{\rho, \kappa/\sqrt{s_k}\}}\right)^2\le1\right\}$. We can easily check that $\tilde{\mathcal{B}}\subset\tilde{\mathcal{F}}_1\subset\sqrt{2}\tilde{\mathcal{B}}$. Therefore, 
	\[\sup_{\tilde{b}\in\tilde{\mathcal{B}}}|W_n(L_{K}^{\frac12}\tilde{b})|\leq\|W_n\|_{\mathcal{B}(\rho,\kappa)}\leq\sqrt{2}\sup_{\tilde{b}\in\tilde{\mathcal{B}}}|W_n(L_K^{\frac12}\tilde{b})|.\]
	By Jensen's inequality,
	\[\mathbb{E}\sup_{\tilde{b}\in\tilde{\mathcal{B}}}|W_n(L_{K}^{\frac12}\tilde{b})|\leq\left(\mathbb{E}\sup_{\tilde{b}\in\tilde{\mathcal{B}}}|W_n(L_{K}^{\frac12}\tilde{b})|^2\right)^{\frac12}.\]
	By Cauchy-Schwartz inequality,
	\[\begin{split}
	|W_n(L_{K}^{\frac12}\tilde{b})|^2&=|\sum_{k\ge1}\tilde{b}_kW_n(L_{K}^{\frac12}\varphi_k)|^2\\
	&\leq\left(\sum_{k\ge1}\frac{b_k^2}{\min\{\rho^2, \kappa^2/s_k\}}\right)\left(\sum_{k\ge1}\min(\rho^2, \kappa^2/s_k)W_n^2(L_{K}^{\frac12}\varphi_k)\right).
	\end{split}\]
	Therefore,
	\[\sup_{\tilde{b}\in\tilde{\mathcal{B}}}|W_n(L_{K}^{\frac12}\tilde{b})|^2\leq\left(\sum_{k\ge1}\min(\rho^2, \kappa^2/s_k)W_n^2(L_{K}^{\frac12}\varphi_k)\right).\]
	Observe that $\mathbb{E}\;W_n^2(L_{K}^{\frac12}\varphi_k)=\mathbb{E}\left(\frac{1}{n}\sum_{i=1}^n\omega_i\langle X_i, L_{K}^{\frac12}\varphi_k\rangle_{\mathcal{L}_2}\right)^2=\frac{1}{n}\mathbb{E}\langle X_i, L_{K}^{\frac12}\varphi_k\rangle^2_{\mathcal{L}_2}=\frac1n s_k.$ This equality also holds for i.i.d Gaussian random variables $\epsilon'_i$s. Thus
	\[\mathbb{E}(\sup_{\tilde{b}\in\tilde{\mathcal{B}}}|W_n(L_{K}^{\frac12}\tilde{b})|^2)\leq\left(\sum_{k\ge1}\min\{\rho^2, \kappa^2/s_k\}\mathbb{E}\;W_n^2(L_{K}^{\frac12}\varphi_k)\right)=\frac{1}{n}\sum_{k\ge1}\min\{\rho^2s_k, \kappa^2\},\]
	which implies that $\mathbb{E}\|W_n\|_{\mathcal{B}(\rho, \kappa)}\leq\sqrt{2}\left(\frac{1}{n}\sum_{k\ge1}\min\{\rho^2s_k, \kappa^2\}\right)^{\frac12}$. Note that \[\begin{split}
	\frac{1}{n}\sum_{k\ge1}\min\{\rho^2s_k, \kappa^2\}&=\frac1n\min_{k_0\ge0}\left(\rho^2\sum_{k\ge k_0}s_k+\kappa^2k_0\right)\\
	&\leq \frac1n\min_{k_0\ge0}\left(c_1\rho^2k_0^{-2r+1}+\kappa^2k_0\right)\\
	&=\frac1n2c^{\frac{1}{2r}}_1\rho^{\frac1{r}}\kappa^{2-\frac1r},
	\end{split}\]
	where the inequality follows from the assumption $s_k\le c_1k^{-2r}$ and the last equality follows by taking $k_0=c^{\frac{1}{2r}}_1\left(\frac{\rho}{\kappa}\right)^{\frac1{r}}.$ Therefore, $\mathbb{E}\|W_n\|_{\mathcal{B}(\rho, \kappa)}\leq \frac1{\sqrt{n}}2c^{\frac{1}{4r}}_1\rho^{\frac1{2r}}\kappa^{1-\frac1{2r}}$.
\end{proof}

\begin{lemma} \label{lemmatech2} Suppose the eigenvalues $\{\tilde{s}_k:k\ge1\}$ of the integral operator $L_G$ satisfy $\tilde{s}_k\leq c_2k^{-2m}$ for some constant $0<m<\infty$. Consider the following Rademacher type of process:
	\[\tilde{W}_n(g)=\frac1n\sum_{i=1}^{n}\omega_i g(\bm{Z}_i),\]
	where $\omega'_i$s are i.i.d. Rademacher random variable, i.e., $\mathbb{P}(\omega_i=1)=\mathbb{P}(\omega_i=-1)=\frac12$. Define
	\[\|\tilde{W}_n\|_{\mathcal{G}_1(\rho, \kappa)}=\sup_{g\in\mathcal{G}(\rho, \kappa)}|\tilde{W}_n(g)|,\]
	where $\mathcal{G}_1(\rho, \kappa)=\{g\in\mathcal{H}_G: \|b\|_{G}\le \rho \hbox{ and } \|g\|\leq \kappa\}$. Then we have \[\mathbb{E}\|\tilde{W}_n\|_{\mathcal{G}_1(\rho, \kappa)}\leq \frac1{\sqrt{n}}2c^{\frac{1}{4m}}_2\rho^{\frac1{2m}}\kappa^{1-\frac1{2m}}.\]
	Note that this lemma also holds for Gaussian type of process:
		\[\tilde{G}_n(g)=\frac1n\sum_{i=1}^{n}\epsilon_ig(\bm{Z}_i),\]
	where $\epsilon'_i$s are i.i.d. standard Gaussian variables.
\end{lemma}
\begin{proof}
	It is clear that $\mathcal{G}_1(\rho, \kappa)=L_{G}^{\frac12}(\tilde{\mathcal{F}}_2(\rho, \kappa))$, where
	\[\tilde{\mathcal{F}}_2(\rho, \kappa)=\{\tilde{f}\in \mathcal{L}^2_{\rho_z}: \|\tilde{f}\|_{\mathcal{L}^2_{\rho_z}}\le \rho \hbox{ and } \|L_{G}^{\frac12}\tilde{f}\|\leq \kappa\}.\]
	Denote $\tilde{\mathcal{G}}_1=\left\{\sum_{k\ge1}\tilde{g}_k\tilde{\varphi}_k: \sum_{k\ge1}\left(\frac{\tilde{g}_k}{\min\{\rho, \kappa/\sqrt{\tilde{s}_k}\}}\right)^2\le1\right\}$. We can easily check that $\tilde{\mathcal{G}}_1\subset\tilde{\mathcal{F}}_2\subset\sqrt{2}\tilde{\mathcal{G}}_1$. Therefore, 
	\[\sup_{\tilde{g}\in\tilde{\mathcal{G}}_1}|\tilde{W}_n(L_{G}^{\frac12}\tilde{g})|\leq\|\tilde{W}_n\|_{\mathcal{G}_1(\rho,\kappa)}\leq\sqrt{2}\sup_{\tilde{g}\in\tilde{\mathcal{G}}_1}|\tilde{W}_n(L_{G}^{\frac12}\tilde{g})|.\]
	By Jensen's inequality,
	\[\mathbb{E}\sup_{\tilde{g}\in\tilde{\mathcal{G}}_1}|\tilde{W}_n(L_{G}^{\frac12}\tilde{g})|\leq\left(\mathbb{E}\sup_{\tilde{g}\in\tilde{\mathcal{G}}_1}|\tilde{W}_n(L_{G}^{\frac12}\tilde{g})|^2\right)^{\frac12}.\]
	By Cauchy-Schwartz inequality,
	\[\begin{split}
	|\tilde{W}_n(L_{G}^{\frac12}\tilde{g})|^2&=|\sum_{k\ge1}\tilde{g}_k\tilde{W}_n(L_{G}^{\frac12}\tilde{\varphi}_k)|^2\leq\left(\sum_{k\ge1}\frac{\tilde{g}_k^2}{\min\{\rho^2, \kappa^2/\tilde{s}_k\}}\right)\left(\sum_{k\ge1}\min(\rho^2, \kappa^2/\tilde{s}_k)\tilde{W}_n^2(L_{G}^{\frac12}\tilde{\varphi}_k)\right).
	\end{split}\]
	Therefore,
	\[\sup_{\tilde{g}\in\tilde{\mathcal{G}}_1}|\tilde{W}_n(L_{G}^{\frac12}\tilde{g})|^2\leq\left(\sum_{k\ge1}\min(\rho^2, \kappa^2/\tilde{s}_k)\tilde{W}_n^2(L_{G}^{\frac12}\tilde{\varphi}_k)\right).\]
	Observe that $\mathbb{E}\;\tilde{W}_n^2(L_{G}^{\frac12}\tilde{\varphi}_k)=\mathbb{E}\left(\frac{1}{n}\sum_{i=1}^n\omega_i L_{G}^{\frac12}\tilde{\varphi}_k\right)^2=\frac{1}{n}\mathbb{E}[\omega^2_i]\mathbb{E}[(L_{G}^{\frac12}\tilde{\varphi}_k)^2] =\frac{1}{n}\mathbb{E}[(\tilde{s}^{\frac{1}{2}}_k \tilde{\varphi}_k)^2]=\frac1n \tilde{s}_k.$ This equality also holds for i.i.d. Gaussian random variables $\epsilon'_i$s. Thus
	\[\mathbb{E}\;\sup_{\tilde{g}\in\tilde{\mathcal{G}}_1}|\tilde{W}_n(L_{G}^{\frac12}\tilde{g})|^2\leq\left(\sum_{k\ge1}\min\{\rho^2, \kappa^2/\tilde{s}_k\}\mathbb{E}\;\tilde{W}_n^2(L_{G}^{\frac12}\tilde{\varphi}_k)\right)=\frac{1}{n}\sum_{k\ge1}\min\{\rho^2\tilde{s}_k, \kappa^2\},\]
	which implies that $\mathbb{E}\|\tilde{W}_n\|_{\mathcal{G}_1(\rho, \kappa)}\leq\sqrt{2}\left(\frac{1}{n}\sum_{k\ge1}\min\{\rho^2\tilde{s}_k, \kappa^2\}\right)^{\frac12}$. Note that \[\begin{split}
	\frac{1}{n}\sum_{k\ge1}\min\{\rho^2\tilde{s}_k, \kappa^2\}&=\frac1n\min_{k_0\ge0}\left(\rho^2\sum_{k\ge k_0}\tilde{s}_k+\kappa^2k_0\right)\\
	&\leq \frac1n\min_{k_0\ge0}\left(c_2\rho^2k_0^{-2m+1}+\kappa^2k_0\right)\\
	&=\frac1n2c^{\frac{1}{2m}}_2\rho^{\frac1{m}}\kappa^{2-\frac1m},
	\end{split}\]
	where the inequality follows from the assumption $\tilde{s}_k\le c_2k^{-2m}$, and the last equality follows by taking $k_0=c^{\frac{1}{2m}}_2\left(\frac{\rho}{\kappa}\right)^{\frac1{m}}.$ Therefore, $\mathbb{E}\|\tilde{W}_n\|_{\mathcal{G}_1(\rho, \kappa)}\leq \frac1{\sqrt{n}}2c^{\frac{1}{4m}}_2\rho^{\frac1{2m}}\kappa^{1-\frac1{2m}}$.
\end{proof}
\subsection{Representer theorem} \label{representersection}
We suppose that an independent random sample $\big\{X_i,\bm{Z}_i,Y_i\big\}^{n}_{i=1}$ is generated from the semi-functional partial linear model in (\ref{formula1}). The double penalized least square estimator of $(\hat{\beta},\hat{g})$ is defined as
\begin{equation} \label{defbetag}
(\hat{\beta},\hat{g})=\displaystyle{\mathop{\arg\min}_{\beta \in \mathcal{H}_K, g \in  \mathcal{H}_G} \left\{\frac{1}{n}\sum_{i=1}^{n} \left(Y_i-\int_{\mathcal{T}} X_i(t)\beta(t)dt-g(\bm{Z}_i)\right)^2+\lambda^2 \lVert \beta \rVert^2_{K}+\xi^2 \lVert g\rVert^2_{G}\right\}}.
\end{equation}
Note that, in (\ref{defbetag}), we define the estimator $(\hat{\beta},\hat{g})$ as the solutions to an infinite-dimensional minimization problem. Next, we show that this infinite-dimensional problem can be reduced to a finite-dimensional optimization problem by virtue of the following representer theorem. Throughout the paper, we write
\begin{equation}
	(L_KX)(\cdot)=\int_\mathcal{T}K(\cdot, s)X(s)ds.
\end{equation}
As mentioned in \cite{FLR2}, $(L_KX)\in \mathcal{H}_K$ for any $X\in \mathcal{L}_2$. Moreover, for any $ \beta\in \mathcal{H}_K$, we have
\begin{equation} \label{observation}
	\int_\mathcal{T} \beta(t)X(t) dt=\langle L_KX,\beta  \rangle_K
\end{equation}
With the above observation, we can prove the following representer theorem which is essential for developing our computational algorithm.
\begin{lemma}\label{representer}
	Let $\hat{\beta}$ and $\hat{g}$ be the minimizer of (\ref{defbetag}), and $\hat{\beta}\in\mathcal{H}_K$, $\hat{g}\in \mathcal{H}_G$. Then there exists $\bm{c}=(c_1,...,c_n)^T\in \mathbb{R}^n$  such that
	\begin{equation} 
		\hat{\beta}(t)=\sum_{i=1}^{n}c_i(L_KX_i)(t).	
	\end{equation}
	Moreover, there exists $\bm{\eta}=(\eta_1,...,\eta_n)^T\in \mathbb{R}^n$ such that
	\begin{equation}
		\hat{g}(\bm{Z})=\sum_{i=1}^{n}\eta_iG(\bm{Z},\bm{Z}_i)
	\end{equation}	
\end{lemma}
\begin{proof}
	According to the observations in Section \ref{representersection}, we have $L_KX_i\in\mathcal{H}_K,\; i=1,...,n$. Given $X_1,...,X_n$, any $\beta\in\mathcal{H}_K$ can be expressed as
	\begin{equation}\label{beta}
	\beta(t)=\sum_{i=1}^{n}c_i(L_KX_i)(t)+v_1(t),
	\end{equation}
	where $v_1\in\mathcal{H}_K$ is the orthogonal complement of $\sum_{i=1}^{n}c_i(L_KX_i)(t)$ in $\mathcal{H}_K$. Given $\bm{Z}_1,...,\bm{Z}_n$, any $g\in\mathcal{H}_G$ can be expressed as
	\begin{equation} \label{g}
	\hat{g}(\bm{Z})=\sum_{i=1}^{n}\eta_iG(\bm{Z},\bm{Z}_i)+v_2(\bm{Z}),
	\end{equation}
	and $v_2\in\mathcal{H}_G$ is the orthogonal complement of $\sum_{i=1}^{n}\eta_iG(\cdot,\bm{Z}_i)$ in $\mathcal{H}_G$. In order to prove the representer theorem, we need to verify $v_1=0$ and $v_2=0$. Substituting (\ref{beta}) and (\ref{g}) into (\ref{defbetag}), we obtain
	\begin{equation*}
	\begin{split}
	\int_{\mathcal{T}} X_i(t)\beta(t)dt&=\int_{\mathcal{T}} X_i(t)\left(\sum_{j=1}^{n}c_j(L_KX_j)(t)+v_1(t)\right)dt\\
	&=	\int_{\mathcal{T}} X_i(t)\sum_{j=1}^{n}c_j(L_KX_j)(t)dt +\int_{\mathcal{T}} X_i(t)v_1(t)dt\\
	&=\int_{\mathcal{T}} X_i(t)\sum_{j=1}^{n}c_j(L_KX_j)(t)dt+\langle L_KX_i,v_1\rangle_K\\
	&=\int_{\mathcal{T}} X_i(t)\sum_{j=1}^{n}c_j(L_KX_j)(t)dt,
	\end{split}
	\end{equation*}
	where the third equality is obtained by (\ref{observation}), and the last equality follows from the fact that $v_1\in\mathcal{H}_K$ is orthogonal to $(KX_i)(t)$. Therefore, this term does not depend on $v_1$. In addition, since $v_1\in\mathcal{H}_K$ is orthogonal to $\sum_{i=1}^{n}c_i(L_KX_i)(t)$ in $\mathcal{H}_K$, we have
	\[\lambda^2 \lVert \beta \rVert^2_{K}=\lambda^2\left\lVert\sum_{i=1}^{n}c_i(L_KX_i)+v_1\right\rVert^2_{K}=\lambda^2\left\lVert\sum_{i=1}^{n}c_i(L_KX_i)\right\rVert^2_{K}+\lambda^2\left\lVert v_1\right\rVert^2_{K}.\]
	This term is minimized when $v_1=0$. Meanwhile, by the orthogonal decomposition of $g(\bm{Z}_i)$ and the reproducing property, we have
	\begin{equation*}
	\begin{split}
	g(\bm{Z}_i)&=\langle g,G(\cdot,\bm{Z}_i)\rangle_G\\
	&=\left\langle\sum_{j=1}^{n}\eta_jG(\cdot,\bm{Z}_j)+v_2,G(\cdot,\bm{Z}_i)\right\rangle_G\\
	&=\sum_{j=1}^{n}\eta_j\langle G(\cdot,\bm{Z}_j),G(\cdot,\bm{Z}_i)\rangle_G+\langle v_2,G(\cdot,\bm{Z}_i)\rangle_G\\
	&=\sum_{j=1}^{n}\eta_j G(\bm{Z}_i,\bm{Z}_j),
	\end{split}
	\end{equation*}
	where the last equality follows from the fact that $v_2\in\mathcal{H}_G$ is orthogonal to $G(\cdot,\bm{Z}_i)$. Therefore, this term does not depend on $v_2$. In addition, since $v_2\in\mathcal{H}_G$ is orthogonal to $\sum_{i=1}^{n}\eta_iG(\cdot,\bm{Z}_i)$ in $\mathcal{H}_G$, we have
	\[\xi^2 \lVert g\rVert^2_{G}=\xi^2\left\lVert\sum_{i=1}^{n}\eta_iG(\cdot,\bm{Z}_i)+v_2\right\rVert^2_{G}=\xi^2\left\lVert\sum_{i=1}^{n}\eta_iG(\cdot,\bm{Z}_i) \right\rVert^2_{G}+\xi^2\left\lVert v_2\right\rVert^2_{G}.\]
	This term is minimized when $v_2=0$. Consequently, the minimizers of (\ref{defbetag}) take the following forms
	\begin{equation*} 
	\hat{\beta}(t)=\sum_{i=1}^{n}c_i(L_KX_i)(t) \quad \hbox{ and }\quad
	\hat{g}(\bm{Z})=\sum_{i=1}^{n}\eta_iG(\bm{Z},\bm{Z}_i).
	\end{equation*}
	This completes the proof.
\end{proof}


\subsection{Computational algorithm} \label{algorithm}
The representer theorem indicates that the solution can be obtained in a finite-dimensional subspace even if the minimization problem is based on a infinite-dimensional space. As a result, we only need to estimate the coefficients $\bm{c}$ and $\bm{\eta}$. In this section, we provide an algorithm to evaluate the coefficients $\bm{c}$ and $\bm{\eta}$. We aim to solve the following minimization problem 
\begin{equation} \label{problem1}
	\min_{\beta \in \mathcal{H}_K, g \in  \mathcal{H}_G} \left\{\frac{1}{n}\sum_{i=1}^{n} \left(Y_i-\int_{\mathcal{T}} X_i(t)\beta(t)dt-g(\bm{Z}_i)\right)^2+\lambda^2 \lVert \beta \rVert^2_{K}+\xi^2 \lVert g\rVert^2_{G}\right\}.
\end{equation}
By Lemma \ref{representer}, it is sufficient to consider $\beta(t)$ with the following form
$$\beta(t)=\sum_{i=1}^{n}c_i\int_{\mathcal{T}}X_i(t)K(t,s)ds,$$
for some $\bm{c}=(c_1,...,c_n)^T\in \mathbb{R}^n$. By Lemma \ref{representer}, we can write the solution of (\ref{problem1}) in terms of $g$ as
$$g(\bm{Z})=\sum_{i=1}^{n} \eta_i G(\bm{Z}_i, \bm{Z})$$
for some $\bm{\eta}=(\eta_1,...,\eta_n)^T\in \mathbb{R}^n$. We denote $\mathbf{G}$ as a $n\times n$ matrix, where the $(i,j)$-th entry is $G(\bm{Z}_i,\bm{Z}_j)$. Then we can reformulate (\ref{problem1}) as 
\begin{equation} \label{problem2}
	\min_{\bm{c} \in \mathbb{R}^n, \bm{\eta} \in \mathbb{R}^{n}} \left\{ \left\lVert \bm{Y}-\mathbf{\Sigma} \bm{c}-\mathbf{G}\bm{\eta}\right\rVert^2_n+\lambda^2 \bm{c}^T\mathbf{\Sigma}\bm{c}+\xi^2\bm{\eta}^T\mathbf{G}\bm{\eta}\right\}.
\end{equation}
where $\mathbf{\Sigma}$ is a $n\times n$ matrix with $\Sigma_{ij}=\int_{\mathcal{T}}\int_{\mathcal{T}}X_i(t)K(t,s)X_j(s)dsdt$.

Minimizing (\ref{problem2}) with respect to $\bm{c}$, we obtain
\begin{equation} \label{c}
	\begin{split}
		\hat{\bm{c}}=(\mathbf{\Sigma}^T\mathbf{\Sigma}+n\lambda^2\mathbf{\Sigma})^{-1}\mathbf{\Sigma}^T(\bm{Y}-\mathbf{G}\bm{\eta}).
	\end{split}
\end{equation}
Then we have
$\int_{\mathcal{T}}\bm{X}(t)\hat{\beta}(t)dt=\mathbf{\Sigma}\hat{\bm{c}}=\mathbf{\Sigma}(\mathbf{\Sigma}^T\mathbf{\Sigma}+n\lambda^2\mathbf{\Sigma})^{-1}\mathbf{\Sigma}^T(\bm{Y}-\mathbf{G}\bm{\eta})	=\mathbf{H}_{\lambda}(\bm{Y}-\mathbf{G}\bm{\eta})$,
where $\mathbf{H}_{\lambda}:=\mathbf{\Sigma}(\mathbf{\Sigma}^T\mathbf{\Sigma}+n\lambda^2\mathbf{\Sigma})^{-1}\mathbf{\Sigma}^T$. Now we can reformulate (\ref{problem2}) as
\begin{equation} \label{problem3}
	\min_{\bm{\eta} \in \mathbb{R}^{n}} \left\{ \left\lVert (\mathbf{I}-\mathbf{H}_{\lambda})(\bm{Y}-\mathbf{G}\bm{\eta})\right\rVert^2_n +\xi^2\bm{\eta}^T\mathbf{G}\bm{\eta}\right\}.
\end{equation}
To solve (\ref{problem3}), we minimize (\ref{problem3}) with respect to $\bm{\eta}$ and obtain
\begin{equation} \label{etahat}
	\hat{\bm{\eta}}=\left[\mathbf{G}^T(\mathbf{I}-\mathbf{H}_{\lambda})^T(\mathbf{I}-\mathbf{H}_{\lambda})\mathbf{G}+n\xi^2\mathbf{G}\right]^{-1}\mathbf{G}^T(\mathbf{I}-\mathbf{H}_{\lambda})^T(\mathbf{I}-\mathbf{H}_{\lambda})\bm{Y}.
\end{equation}
Then we substitute $\hat{\bm{\eta}}$ into (\ref{c}) and obtain $\hat{\bm{c}}$. Now we can get the regularized estimators
\begin{equation*}
	\begin{split}
		&\hat{\beta}(t)=\sum_{i=1}^{n}\hat{c}_i\int_{\mathcal{T}}X_i(t)K(t,s)ds,\\
		&\hat{g}(\bm{Z})=\sum_{i=1}^{n} \hat{\eta}_i G(\bm{Z}_i, \bm{Z}).
	\end{split}
\end{equation*}
The procedure of solving the minimization problem (\ref{problem1}) is summarized in Algorithm \ref{algorithm1}.

\begin{algorithm}[htb]
	\caption{Solution to the optimization problem in (\ref{problem1}).} 
	\label{alg:Framwork} 
	\begin{algorithmic}[1] 
		\REQUIRE ~~\\ 
		Responses $\bm{Y}=(Y_1,...,Y_n)^T$;\\
		Functional covariate $X_1(t),....,X_n(t)$ and vector covariate $\bm{Z}_1,....,\bm{Z}_n$;\\
		Regularization parameters $\lambda$ and $\xi$;
		\ENSURE ~~\\ 
		\STATE Computing $\mathbf{\Sigma}$ and $\mathbf{G}$; 
		\STATE $\mathbf{H}_{\lambda}\leftarrow \mathbf{\Sigma}(\mathbf{\Sigma}^T\mathbf{\Sigma}+n\lambda^2\mathbf{\Sigma})^{-1}\mathbf{\Sigma}^T$; 
		\STATE $\hat{\bm{\eta}}\leftarrow\left[\mathbf{G}^T(\mathbf{I}-\mathbf{H}_{\lambda})^T(\mathbf{I}-\mathbf{H}_{\lambda})\mathbf{G}+n\xi^2\mathbf{G}\right]^{-1}\mathbf{G}^T(\mathbf{I}-\mathbf{H}_{\lambda})^T(\mathbf{I}-\mathbf{H}_{\lambda})\bm{Y}$; 
		\STATE 
		$\hat{\bm{c}}\leftarrow(\mathbf{\Sigma}^T\mathbf{\Sigma}+n\lambda^2\mathbf{\Sigma})^{-1}\mathbf{\Sigma}^T(\bm{Y}-\mathbf{G}\hat{\bm{\eta}})$; 
		\RETURN $\hat{\bm{\eta}}$ and $\hat{\bm{c}}$; 
	\end{algorithmic}
	\label{algorithm1}
\end{algorithm}

\section{Main results}  \label{results}
In order to present the theoretical results for the proposed estimators, we first define

\begin{equation*}
	\mathcal{F}:=\left\{ f(X,\bm{Z})=\int_{\mathcal{T}} X(t)\beta(t)dt+g(\bm{Z}), \beta\in \mathcal{H}_K, g \in \mathcal{H}_G\right\},
\end{equation*}
and
\begin{equation*}
	f^{0}(X,\bm{Z}):=\int_{\mathcal{T}} X(t)\beta^{0}(t)dt+g^0(\bm{Z}).
\end{equation*}
Moreover, we let
\begin{equation} \label{deftau}
\tau(f) :=\lVert f \rVert+\lambda \lVert\beta\rVert_{K}+\xi  \lVert g\rVert_{G}   ,
\end{equation}
and 
\begin{equation}\label{deftauRI}
\tau_I(\beta) :=\left\lVert \int_{\mathcal{T}} X(t)\beta(t)dt  \right\rVert+\lambda \lVert\beta\rVert_{K}.
\end{equation}
Next, we introduce a set of technical assumptions in this paper. 
\begin{assumption} \label{error}(Gaussian assumption) The errors $\epsilon_1,..., \epsilon_n$ are i.i.d. standard Gaussian random variables and indepenent of $\left\{X_i,\bm{Z}_i\right\}^{n}_{i=1}$. 	
\end{assumption}

Gaussian assumption is quite restrictive, but it is common in literature, see \cite{PLM} and \cite{Partialfunctional7}. We can always relax this assumption to sub-Gaussian errors. For the sake of simplicity, we assume a unit variance for the Guassian errors. However, the variance is usually unknown in reality. We will not discuss the variance problem here since it is out of the scope of this paper. Readers may refer to \cite{unknownvariance1} and \cite{unknownvariance2} for several existing methods of coping with unknown variance. 

\begin{assumption} \label{design} (Design assumption) 
	(1) $\lVert X \rVert_{\mathcal{L}_2} \leq D_x < \infty$ almost surely. (2) $D_K:=\sup_{t\in \mathcal{T}}\left|K(t,t)\right|<\infty$. (3) $D_G:=\sup_{\bm{z}\in \mathbb{R}^p}\left|G(\bm{z},\bm{z})\right|<\infty$.
\end{assumption}

Although (1) in Assumption \ref{design} seems to be a quite strong assumption, a non-bounded distribution or process can usually be approximated with its truncated version. (2) and (3) in Assumption \ref{design} are not restrictive. The frequently used Gaussian kernel satisfies this kind of assumption.

\begin{assumption} \label{smoothness} (Eigenvalue decay assumption)
	(1) For the eigenvalues $\{s_k: {k\ge1}\}$ of the integral operator $\mathscr{L}$ and some constant $0<r<\infty$, it holds that $s_k\leq c_1k^{-2r}$ . (2) For the eigenvalues $\{\tilde{s}_k: {k\ge1}\}$ of the integral operator $L_G$ and some constant $0<m<\infty$, it holds that $\tilde{s}_k\leq c_2k^{-2m}$ .
\end{assumption}
The smoothness of the functional component is controlled by the rate of decay of the eigenvalues $\{s_k: {k\ge1}\}$, i.e. $r$. The smoothness of $g^0$ is controlled by the rate of decay of the eigenvalues $\{\tilde{s}_k: {k\ge1}\}$, i.e. $m$. Note that the eigenvalues $\{s_k: {k\ge1}\}$ depend on both of covariance function $C$ and the reproducing kernel $K$, so the smoothness of the functional component depends on both $K$ and $C$. As for the nonparametric component, its smoothness only depends on the reproducing kernel $G$ since the eigenvalues $\{\tilde{s}_k: {k\ge1}\}$ only rely on $G$. Without loss of generality, we consider the case where the smoothness of the functional component is larger than the smoothness of the nonparametric component, that is, $r>m$.

In the following, we present the theoretical results we obtained. The overall convergence rate of this method and the convergence rate for the nonparametric component are provided by Theorem \ref{thm1}. A more refined prediction error for the functional component is given by Theorem \ref{thm2}. We provide detailed proofs in Appendix.

\begin{theorem} \label{thm1}
	Under Assumptions \ref{error}, \ref{design} and \ref{smoothness}, we let $L_{\mathcal{T}}$ be a suitable small constant only depending on $D_x$, $D_{K}$, $D_{G}$, $c_1$, $c_2$, $r$ and $m$. Suppose there is a constant R satisfying
	\begin{equation} \label{thm1assump}
		\begin{split}
		&\lambda^2 \lVert \beta^0 \rVert^2_{K}+\xi^2 \lVert g^0 \rVert^2_{G} \leq \delta^2 R^2,\\
		&\lambda \leq \xi \leq R,
		\end{split}
	\end{equation}
	where $\delta >0$ is a small constant. Then $\tau(\hat{f}-f^0)\leq R$ holds with probability at least $1-3\exp\left[-nL^2_{\mathcal{T}}\lambda\right]$.
\end{theorem}	

\begin{theorem} \label{thm2}
	Under the same assumptions of Theorem \ref{thm1}, we let $\L_{\tau_I}$ be a suitable small constant only depending on $D_x$, $D_{K}$, $D_{G}$, $c_1$ and $r$. Suppose there is a constant $R_I$ satisfying 
	\begin{equation} \label{thm2assump}
	\begin{split}
		&\lambda^2 \lVert \beta^0 \rVert^2_{K}\leq \delta^2_I R^2_I,\\
		&\lambda \leq R_I \leq R \leq 1,\\
		&R_I\geq R\frac{\lambda}{\xi},
	\end{split}
  \end{equation}
where constant $R$ satisfies (\ref{thm1assump}) and $\delta_I >0$ is a small constant. It holds that $\tau_I(\hat{\beta}-\beta^0)\leq R_I$ with probability at least $1-7\exp\left[-nL^2_{\mathcal{T}_I}\lambda\right]$.
\end{theorem}

\begin{remark} \label{overallrate}
	Suppose the assumptions in Theorems \ref{thm1} and \ref{thm2} hold. Let 
	$$\lambda^2 \asymp n^{-\frac{2r}{2r+1}},  \:\:\:\:  \xi^2\asymp n^{-\frac{2m}{2m+1}}.$$
	Taking $R^2 \asymp \xi^2$ and $R^2_I \asymp \lambda^2$, an upper bound for the overall rate of convergence for prediction follows from Theorem \ref{thm1}:
	$$\lVert \hat{f}-f^0\rVert^2=\mathcal{O}_p(n^{-2m/(2m+1)}),$$	
	which immediately implies the convergence rate for the nonparametric component:
		$$\lVert \hat{g}-g^0\rVert^2=\mathcal{O}_p(n^{-2m/(2m+1)}).$$	
	Theorem \ref{thm2} provides the convergence rate for functional component:
	$$\left\lVert \int_{\mathcal{T}} X(t)(\hat{\beta}(t)-\beta^0(t))dt  \right\rVert^2=\mathcal{O}_p(n^{-2r/(2r+1)}).$$	
\end{remark}

In summary, with an appropriate choice of parameters $\lambda$ and $\xi$, the prediction error, $\lVert \hat{g}-g^0\rVert^2$, is of order $n^{-2m/(2m+1)}$ in probability. This rate is solely determined by the sample size $n$ and the smoothness of the nonparametric component. The smoother functional component can be learned with the minimax rate. The prediction error, $\left\lVert \int_{\mathcal{T}} X(t)(\hat{\beta}(t)-\beta^0(t))dt\right\rVert^2$, is of order $n^{-2r/(2r+1)}$ in probability. The less smooth function $g^0$ does not affect the rate for the smoother functional component. We can see that, for each component, the rate of convergence can be obtained as if the other component were known.

\section{Numerical studies} \label{simulation}
In this section, simulations are conducted to measure the performance of the proposed double-penalized method and demonstrate the conclusions in Theorem \ref{thm1} and \ref{thm2} also hold in finite sample setting. We consider the case where $t \in\mathcal{T}=[0,1]$ and $z \in\mathcal{Z}=[0,1]$. The scalar variable $Z$ is a univariate variable generated from $U(0,1)$. Without loss of generality, we consider a special case where both $\beta^0$ and $g^0$ reside in the Sobolev space of order 2. Then we can re-write the minimization problem (\ref{defbetag}) as
\begin{equation} \label{simu_problem}
	(\hat{\beta},\hat{g})=\min_{\beta \in \mathcal{H}, g \in  \mathcal{G}} \left\{\frac{1}{n}\sum_{i=1}^{n} \left(Y_i-\int_{\mathcal{T}} X_i(t)\beta(t)dt-g(Z_i)\right)^2+\lambda^2 J(\beta)+\xi^2 I(g)\right\},
\end{equation}
where $\mathcal{H}=\mathcal{W}^2_2$ and $\mathcal{G}=\mathcal{W}^2_2$, and the penalty functionals $J(\beta)$ and $I(g)$ are defined as squared semi-norms on $\mathcal{H}$ and $\mathcal{G}$ respectively, i.e. $J(\beta)=\int_\mathcal{T}(\beta'')^2$ and $I(g)=\int_\mathcal{Z}(g'')^2$. The null space is a finite-dimensional linear subspace of $\mathcal{H}$ such that
 \[\mathcal{H}_0=\left\{ \beta\in\mathcal{H}: J(\beta)=0 \right\}.\]
We let $N=\dim(\mathcal{H}_0)$ and denote the orthonormal basis functions by $\left\{\Phi_1,...,\Phi_N\right\}$. In this case, we have $N=2$, and $\mathcal{H}_0$ is the linear space spanned by basis functions $\Phi_1(t)=1$ and $\Phi_2(t)=t$. Denote by $\mathcal{H}_1$ the orthogonal complement of $\mathcal{H}_0$ in $\mathcal{H}$ such that $\mathcal{H}=\mathcal{H}_0\oplus\mathcal{H}_1$. In other words, for any $\beta\in \mathcal{H}$, there exists a unique decomposition $\beta=\beta_0+\beta_1$ with $\beta_0\in\mathcal{H}_0 $ and $\beta_1\in\mathcal{H}_1 $. Note that $\mathcal{H}_1$ is also a reproducing kernel Hilbert space where the inner product of $\mathcal{H}$ is restricted to $\mathcal{H}_1$. We denote $K(\cdot,\cdot)$ as the reproducing kernel which corresponds to $\mathcal{H}_1$ such that $J(\beta_1)=\lVert \beta_1\rVert^2_K=\lVert \beta_1\rVert^2_\mathcal{H}$ for any $\beta_1\in \mathcal{H}_1$. A commonly used reproducing kernel for $\mathcal{H}_1$ is 
\[K(s,t)=\frac{1}{(2!)^2}B_2(s)B_2(t)-\frac{1}{4!}B_4(|s-t|),\]
where $B_l(\cdot)$ is the $l$-th Bernoulli polynomial. The penalty functional $I(g) $ is a squared semi-norm on $\mathcal{G}$ such that the null space 
\[\mathcal{G}_0=\left\{ \beta\in\mathcal{G}: I(g)=0 \right\}\]
is a finite-dimensional linear subspace of $\mathcal{G}$. We let $\tilde{N}=\dim(\mathcal{G}_0)$ and denote the orthonormal basis functions by $\left\{ \tilde{\Phi}_1,...,\tilde{\Phi}_{\tilde{N}}\right\}$. In this case, we have $\tilde{N}=2$, and $\mathcal{G}_0$ is the linear space spanned by basis functions $\tilde{\Phi}_1(z)=1$ and $\tilde{\Phi}_2(z)=z$. Denote by $\mathcal{G}_1$ the orthogonal complement of $\mathcal{G}_0$ in $\mathcal{G}$ such that $\mathcal{G}=\mathcal{G}_0\oplus\mathcal{G}_1$. In other words, for any $g\in \mathcal{G}$, there exists a unique decomposition $g=g_0+g_1$ with $g_0\in\mathcal{G}_0 $ and $g_1\in\mathcal{G}_1 $. Note that $\mathcal{G}_1$ is also a reproducing kernel Hilbert space where the inner product of $\mathcal{G}$ is restricted to $\mathcal{G}_1$. We denote $G(\cdot,\cdot)$ as the reproducing kernel which corresponds to $\mathcal{G}_1$ such that $I(g_1)=\lVert g_1\rVert^2_G=\lVert g_1\rVert^2_\mathcal{G}$ for any $g_1\in \mathcal{G}_1$. A commonly used reproducing kernel for $\mathcal{G}_1$ is 
\[G(z,w)=\frac{1}{(2!)^2}B_2(z)B_2(w)-\frac{1}{4!}B_4(|z-w|),\]
where $B_l(\cdot)$ is the $l$-th Bernoulli polynomial. In the following, we present the representer theorem (Lemma \ref{representer}) for this special case. The proof is similar to the proof of Lemma \ref{representer}.

\begin{lemma}
Let $\hat{\beta}$ and $\hat{g}$ be the minimizer of (\ref{simu_problem}), and $\hat{\beta}\in\mathcal{H}$, $\hat{g}\in \mathcal{G}$. Then there exist $\bm{d}=(d_1,d_2)^T$ and $\bm{c}=(c_1,...,c_n)^T\in \mathbb{R}^n$  such that
\begin{equation*} 
	\beta(t)=d_1+d_2t+\sum_{i=1}^{n}c_i\int_{\mathcal{T}}X_i(t)K(t,s)ds.
\end{equation*}
Moreover, there exist $\bm{l}=(l_1,l_2)^T$ and $\bm{\eta}=(\eta_1,...,\eta_n)^T\in \mathbb{R}^n$ such that
\begin{equation*}
	g(Z)=l_1+l_2Z+\sum_{i=1}^{n} \eta_i G(Z_i, Z).
\end{equation*}
\end{lemma}
In the following, we adapt Algorithm \ref{algorithm1} proposed in Section \ref{algorithm} to this special case. We denote $\mathbf{G}$ as a $n\times n$ matrix, in which the $(i,j)$-th entry is $G(Z_i,Z_j)$. Let $\bm{1}$ be a column vector of 1. We define $\bm{\theta}:=(l_1,l_2,\bm{\eta}^T)^T\in \mathbb{R}^{n+2}$, $\mathbf{N}:=(\bm{1},\bm{Z},\mathbf{G})\in \mathbb{R}^{n\times(n+2)}$ and 
$$\mathbf{L}:=\left[ 
\begin{matrix}
	\mathbf{0}_{2\times2} & \mathbf{0}_{2\times n} \\
	\mathbf{0}_{n\times2} & \mathbf{G}_{ n\times n}
\end{matrix}
\right],$$Then we can reformulate (\ref{problem1}) as 
\begin{equation} \label{simu_problem2}
	\min_{\bm{d} \in \mathbb{R}^2, \bm{c} \in \mathbb{R}^n, \bm{\theta} \in \mathbb{R}^{n+2}} \left\{ \left\lVert \bm{Y}-(\mathbf{A}\bm{d}+\mathbf{\Sigma} \bm{c})-\mathbf{N}\bm{\theta}\right\rVert^2_n+\lambda^2 \bm{c}^T\mathbf{\Sigma}\bm{c}+\xi^2\bm{\theta}^T\mathbf{L}\bm{\theta}\right\}.
\end{equation}
where $\mathbf{\Sigma}$ is a $n\times n$ matrix with $\Sigma_{ij}=\int_{\mathcal{T}}\int_{\mathcal{T}}X_i(t)K(t,s)X_j(s)dsdt$, and $\mathbf{A}$ is a $n\times 2$ matrix with $A_{im}=\int_{\mathcal{T}}X_i(t)t^{m-1}dt, \; m=1,2.$ Minimizing (\ref{simu_problem2}) with respect to $\bm{d}$ and $\bm{c}$, we obtain
\begin{equation}\label{dandc}
	\begin{split}
		&\hat{\bm{d}}=(\mathbf{A}^T\mathbf{W}^{-1}\mathbf{A})^{-1}\mathbf{A}^{T}\mathbf{W}^{-1}(\bm{Y}-\mathbf{N}\bm{\theta}),\\
		&\hat{\bm{c}}=\mathbf{W}^{-1}[I-\mathbf{A}(\mathbf{A}^T\mathbf{W}^{-1}\mathbf{A})^{-1}\mathbf{A}^T\mathbf{W}^{-1}](\bm{Y}-\mathbf{N}\bm{\theta}),
	\end{split}	
\end{equation}
where $\mathbf{W}=\mathbf{\Sigma}+n\lambda^2 \mathbf{I}$. Then we have
\begin{equation*}
	\begin{split}
		\int_{\mathcal{T}}\bm{X}(t)\hat{\beta}(t)dt&=\mathbf{A}\hat{\bm{d}}+\mathbf{\Sigma}\hat{\bm{c}}\\
		&=[\mathbf{A}(\mathbf{A}^T\mathbf{W}^{-1}\mathbf{A})^{-1}\mathbf{A}^{T}\mathbf{W}^{-1}+\mathbf{\Sigma} \mathbf{W}^{-1}(\mathbf{I}-\mathbf{A}(\mathbf{A}^T\mathbf{W}^{-1}\mathbf{A})^{-1}\mathbf{A}^T\mathbf{W}^{-1})](\bm{Y}-\mathbf{N}\bm{\theta})	\\
		&=\mathbf{H}_{\lambda}(\bm{Y}-\mathbf{N}\bm{\theta}),
	\end{split}
\end{equation*}
where $\mathbf{H}_{\lambda}:=\mathbf{A}(\mathbf{A}^T\mathbf{W}^{-1}\mathbf{A})^{-1}\mathbf{A}^{T}\mathbf{W}^{-1}+\mathbf{\Sigma} \mathbf{W}^{-1}(\mathbf{I}-\mathbf{A}(\mathbf{A}^T\mathbf{W}^{-1}\mathbf{A})^{-1}\mathbf{A}^T\mathbf{W}^{-1})$. Now we can reformulate (\ref{problem2}) as
\begin{equation} \label{simu_problem3}
	\min_{\bm{\theta} \in \mathbb{R}^{n+2}} \left\{ \left\lVert (\mathbf{I}-\mathbf{H}_{\lambda})(\bm{Y}-\mathbf{N}\bm{\theta})\right\rVert^2_n +\xi^2\bm{\theta}^T\mathbf{L}\bm{\theta}\right\}.
\end{equation}
To solve (\ref{simu_problem3}), we minimize (\ref{simu_problem3}) with respect to $\bm{\theta}$ and obtain
\begin{equation} \label{thetahat}
	\hat{\bm{\theta}}=\left[\mathbf{N}^T(\mathbf{I}-\mathbf{H}_{\lambda})^T(\mathbf{I}-\mathbf{H}_{\lambda})\mathbf{N}+n\xi^2\mathbf{L}\right]^{-1}\mathbf{N}^T(\mathbf{I}-\mathbf{H}_{\lambda})^T(\mathbf{I}-\mathbf{H}_{\lambda})\bm{Y}.
\end{equation}
Now we can get $\hat{\bm{l}}$ and $\hat{\bm{\eta}}$. We substitute $\hat{\bm{\theta}}$ back into (\ref{dandc}) and immediately get $\hat{\bm{d}}$ and $\hat{\bm{c}}$. 

The target slope function $\beta^0(t)$ and the functional predictor $X(t)$ are given by
$$\beta^0(t)=\sum_{k=1}^{50}4(-1)^{k+1}k^{-2}\phi_k(t),\;\;\; X(t)=\sum_{k=1}^{50}\zeta_kX_k\phi_k(t),\;\;\; $$
where $\phi_1(t)=1$, $\phi_{k+1}(t)=\sqrt{2}\cos(k\pi t)$ for $k\geq 1$, $X_k \sim U(-\sqrt{3}, \sqrt{3})$ and $\zeta_k$ are defined as $\zeta_k=(-1)^{k+1}k^{-\upsilon_1/2}$. It is clear that $\zeta^2_k$ are the eigenvalues of covariance function of $X$. A larger $\upsilon_1$ indicates a smoother $X(t)$. The smoothness of $\beta^0$ is fixed. Hence, the parameter $\upsilon_1$ controls the smoothness of the whole functional component. For simplicity, we only consider a univariate scalar predictor. The true $g^0(z)$ is given by
$$g^0(z)=\sum_{k=1}^{50}4(-1)^{k+1}k^{-\upsilon_2}\phi_k(z)$$
where $z \in\mathcal{Z}=[0,1]$, $\phi_1(z)=1$ and $\phi_{k+1}(z)=\sqrt{2}\cos(k\pi z)$ for $k\geq 1$. The scalar random variable $Z_i,\; i=1,...,n$ is generated from $U(0,1)$. The smoothness of $g^0$ is controlled by $\upsilon_2$. A smaller value of $\upsilon_2$ corresponds to a less smooth $g^0$. The response Y are generated by 
\begin{equation} \label{generate}
	Y_i=\int^1_0 X_i(t)\beta^0(t) dt+g^0(Z_i)+\epsilon_i,
\end{equation}
where $\epsilon_i \sim N(0,1)$. 
The minimization problem (\ref{problem1}) involves two tuning parameters $\lambda$ and $\xi$. We can use grid search to select $\lambda$ and $\xi$. We choose the pair $(\lambda,\xi)$ which minimizes the GCV criterion 
$$GCV(\lambda,\xi)=\frac{ \lVert \hat{\bm{Y}}-\bm{Y}\rVert^2_n}{(1-tr(\mathbf{H}_\lambda)/n)^2}.$$
With different choices of $n$, $\upsilon_1$ and $\upsilon_2$, we generate different datasets and apply the aforementioned computing algorithm to each simulated dataset. The performance of estimating the functional component is examined by
\[\lVert\hat{\beta}-\beta^0\rVert^2_{n^*}=\frac{1}{n^*}\sum_{i=1}^{n^*}\left[ \int_{\mathcal{T}} X_i(t) (\hat{\beta}(t) -\beta^0)dt\right]^2,\]
and the performance of estimating the nonparametric component is examined by
\[\lVert\hat{g}-g^0\rVert^2_{n^*}=\frac{1}{n^*}\sum_{i=1}^{n^*}[ \hat{g}(Z_i) -g^0(Z_i)]^2.\] 
Note that 1000 independent testing datasets are used to evaluate the prediction errors, i.e. $n^*=1000$. We conduct simulations under three scenarios. In the first scenario, we consider the case where both $\beta^0$ and $g^0$ are unknown. We explore how smoothness of the functional and nonparametric component affect the prediction errors. In the second scenario, we consider the case where $g^0$ is known but $\beta^0$ is unknown. We compare the results obtained with the first scenario to investigate the influence of availability of $g^0$ on estimating the functional component. In the third scenario, we consider the case where $\beta^0$ is known but $g^0$ is unknown to study the influence of availability of $\beta^0$ on estimating the nonparametric component.

\subsection{The case where both $\beta^0$ and $g^0$ are unknown} \label{bothunknown}
In this case, we assume neither functions are known. Scenarios with $n=50,100,200,500$, $\upsilon_1=1.1,1.5,2,4$ and $\upsilon_2=1.5,1.7,2,4$ are considered. For each scenario, we repeat the experiment for 100 times and record the averaged performance measures. For a better illustration, we plot the averaged prediction errors aginst the sample size $n$ for each choice of $\upsilon_1$ and $\upsilon_2$ (see Figure \ref{Figure1} and \ref{Figure2}). Note that both axaes are in log scale. 

In Figure \ref{Figure1}, we plot the averaged prediction errors, $\lVert\hat{\beta}-\beta^0\rVert^2_{n^*}$ and $\lVert\hat{g}-g^0\rVert^2_{n^*}$, aginst the sample size $n$ when $\upsilon_1=1.1,1.5,2,4$ and $\upsilon_2$ is fixed to be 1.5. The trend of the averaged prediction error, $\lVert\hat{\beta}-\beta^0\rVert^2_{n^*}$, for each choice of $\upsilon_1$ is shown in the left panel of Figure \ref{Figure1}. From the plot, one can observe that the prediction error, $\lVert\hat{\beta}-\beta^0\rVert^2_{n^*}$, gradually converges to 0 as sample size $n$ increases for any $\upsilon_1$. Moreover, the plot suggests that with the same sample size $n$, a larger value of $\upsilon_1$ could lead to a smaller prediction error $\lVert\hat{\beta}-\beta^0\rVert^2_{n^*}$. The right panel of Figure \ref{Figure1} presents the trend of the averaged prediction error, $\lVert\hat{g}-g^0\rVert^2_{n^*}$, for each choice of $\upsilon_1$. We observe that the prediction error of the nonparametric component also converges to 0 as sample size $n$ increases for any $\upsilon_1$. Unlike the prediction error of the functional component, there is no significant difference in the prediction error, $\lVert\hat{g}-g^0\rVert^2_{n^*}$, with different choices of $\upsilon_1$. To summarize, Figure \ref{Figure1} indicates that the convergence rate of the prediction error, $\lVert\hat{\beta}-\beta^0\rVert^2_{n^*}$, is controlled by both of the sample size and the smoothness of the functional component. This prediction error tends to be smaller if the functional component gets smoother. However, the smoothness of the functional component has no impact on the prediction error of the nonparametric component. One can see that the observations from Figure \ref{Figure1} confirms our theories in Section \ref{results}.

\begin{figure} [htb]
	\centering
	\includegraphics[scale=0.55]{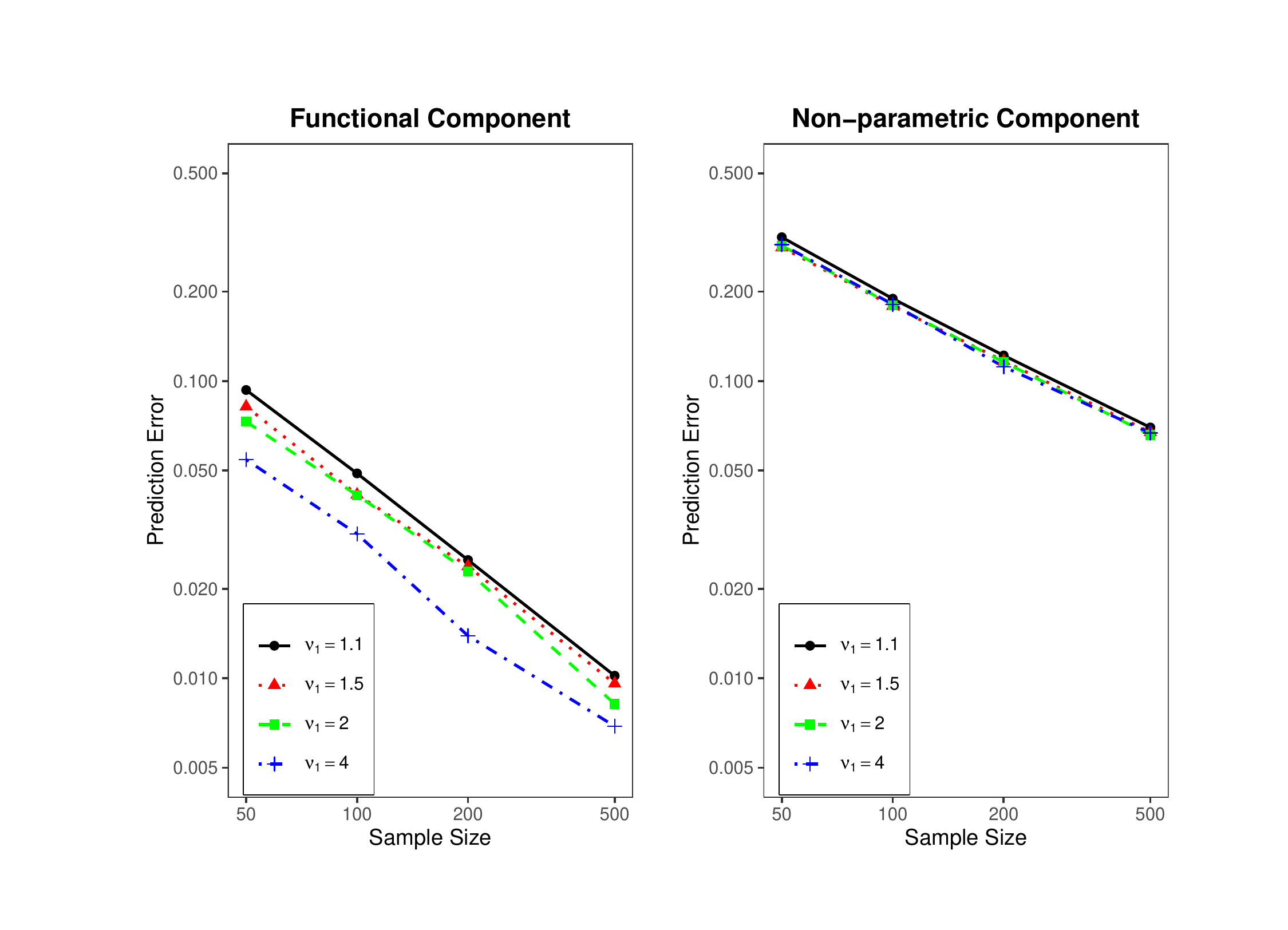}
	\caption{\textit{Prediction errors of the proposed estimators with} $\upsilon_2=1.5$. \textit{Both axes are in log scale. The performance measure is averaged over 100 simulations. The black solid lines, red dotted lines, green dashed lines and blue dot-dashed lines correspond to the scenarios where $\upsilon_1=1.1,\;1.5,\;2$ and $4$ respectively.}}
	\label{Figure1}
\end{figure}

Next, in Figure \ref{Figure2}, we fix the value of $\upsilon_1$ to be 1.1 and change $\upsilon_2$. Cases where $\upsilon_2=1.5,1.7,2,4$ are considered. We plot the averaged prediction errors, $\lVert\hat{\beta}-\beta^0\rVert^2_{n^*}$ and $\lVert\hat{g}-g^0\rVert^2_{n^*}$, aginst the sample size $n$ in the left panel and right panel of Figure \ref{Figure2} respectively. Similarly to Figure \ref{Figure1}, one can observe that both $\lVert\hat{\beta}-\beta^0\rVert^2_{n^*}$ and $\lVert\hat{g}-g^0\rVert^2_{n^*}$ converge to 0 as sample size $n$ increases. From the right panel, we can observe that the prediction error of the nonparametric component gets smaller if the value of $\upsilon_2$ becomes larger with the same sample size. By contrast, the value of $\upsilon_2$ does not affect the prediction error of the functional component. These phenomenon again confirms our theories in Section \ref{results}.

\begin{figure} [htb]
	\centering
	\includegraphics[scale=0.55]{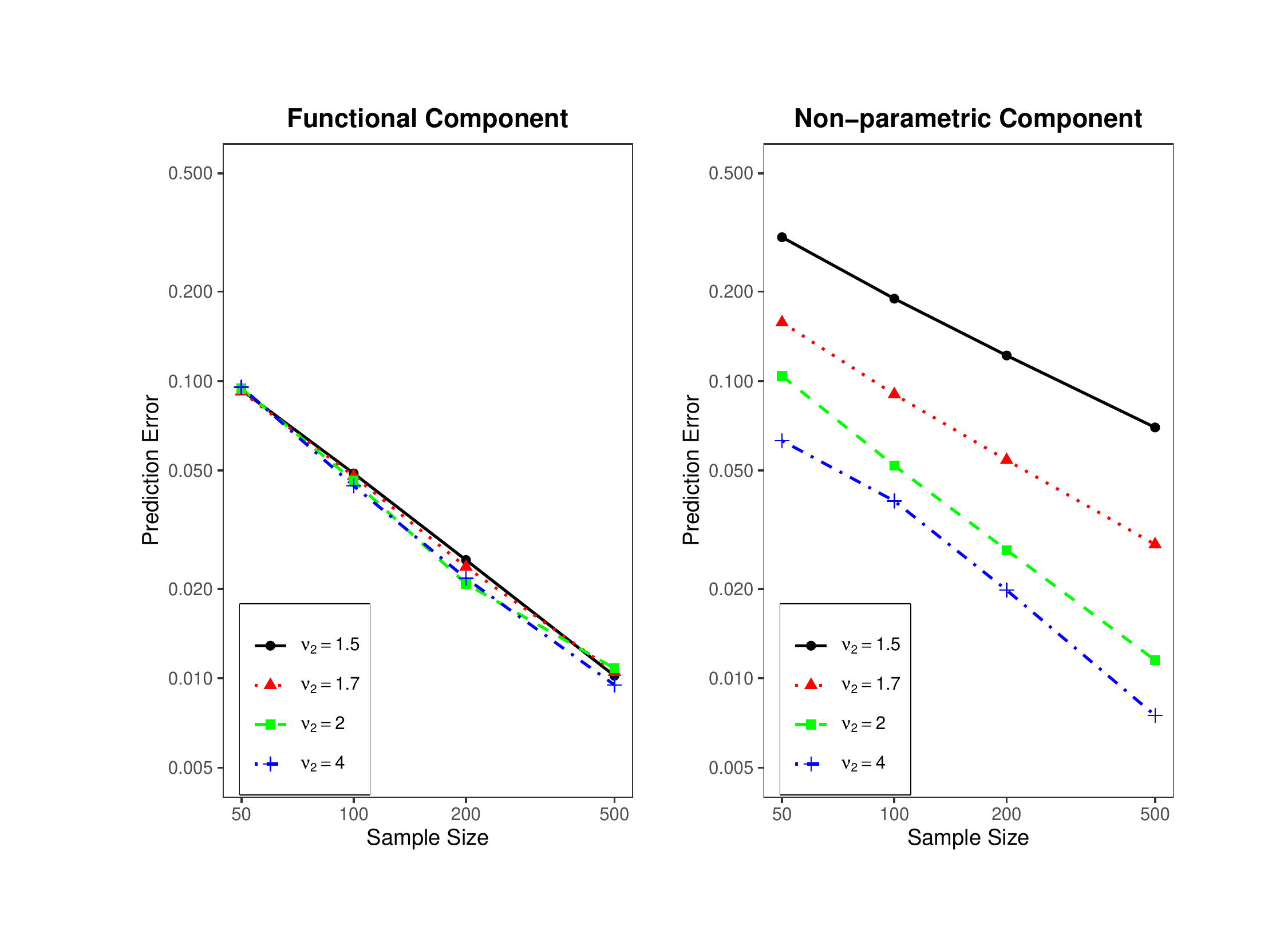}
	\caption{\textit{Prediction errors of the proposed estimators with} $\upsilon_1=1.1$. \textit{Both axes are in log scale. The performance measure is averaged over 100 simulations. The black solid lines, red dotted lines, green dashed lines and blue dot-dashed lines correspond to the scenarios where $\upsilon_2=1.5,\;1.7,\;2$ and $4$ respectively.}}
	\label{Figure2}
\end{figure}

\subsection{The case where $g^0$ is known} \label{g0known}
In this case, we assume that $g^0$ is known and estimate the functional component. If $g^0$ is known, we can re-write (\ref{generate}) as
$$Y^{\beta}_i=\int^1_0 X_i(t)\beta^0(t) dt+\epsilon_i,$$ 
where $Y^{\beta}_i:=Y_i-g^0(Z_i)$. The oracle estimator of $\beta^0$ is
$$\hat{\beta}^{oracle}=\min_{\beta \in \mathcal{H}} \left\{\frac{1}{n}\sum_{i=1}^{n} \left(Y^{\beta}_i-\int_{\mathcal{T}} X_i(t)\beta(t)dt\right)^2+\lambda^2 \lVert \beta \rVert^2_{\mathcal{H}}\right\}.$$
The value of $\upsilon_2$ is fixed and taken to be 1.5 since our simulation results in Section \ref{bothunknown} already show that the smoothness of $g^0$ has no influence on the prediction error $\lVert\hat{\beta}-\beta^0\rVert^2_{n^*}$. Scenarios with $n=50,100,200,500$ and $\upsilon_1=1.1,1.5,2,4$ are considered. For each scenario, we repeat the experiment for 100 times and record the averaged performance measures. In the left panel of Figure \ref{Figure3}, we plot the averaged prediction errors aginst the sample size $n$ for each choice of $\upsilon_1$. Note that both axaes are in log scale. By comparing the left panel of Figure \ref{Figure3} with the left panel of Figure \ref{Figure1}, we find that there is no siginificant difference bewteen these two figures. Given the true $g^0$, we still obtain similar prediction errors for the functional component. In other words, the functional component can be well-estimated even if the true $g^0$ is unknown, which coincides with our theorems.


\subsection{The case where $\beta^0$ is known} \label{beta0known}
In this case, we assume that $\beta^0$ is known and estimate the nonparametric component. If $\beta^0$ is known, we can re-write (\ref{generate}) as
$$Y^g_i=g(Z_i)+\epsilon_i,$$ 
where $Y^g_i:=Y_i-\int^1_0 X_i(t)\beta^0(t) dt$. The oracle estimator of $g^0$ is
$$\hat{g}^{oracle}=\min_{g \in \mathcal{G}} \left\{\frac{1}{n}\sum_{i=1}^{n} \left(Y^g_i-g(Z_i)\right)^2+\xi^2 \lVert g \rVert^2_{\mathcal{G}}\right\}.$$
The value of $\upsilon_1$ is fixed and taken to be 1.1 since our simulation results in Section \ref{bothunknown} already show that the smoothness of the functional component has no influence on the prediction error $\lVert\hat{g}-g^0\rVert^2_{n^*}$. Scenarios with $n=50,100,200,500$ and $\upsilon_2=1.5,1.7,2,4$ are considered. For each scenario, we repeat the experiment for 100 times and record the averaged performance measures. In the right panel of Figure \ref{Figure3}, we plot the averaged prediction errors aginst the sample size $n$ for each choice of $\upsilon_2$. Note that both axaes are in log scale. We compare the right panel of Figure \ref{Figure3} with the right panel of Figure \ref{Figure2} and find that the rate of convergence of $\hat{g}$ and $\hat{g}^{oracle}$ are of similar order for each choice of $\upsilon_2$. Therefore, one can see that the performance of estimating $g^0$ does not improve by incorporating the information of true $\beta^0$, which is in agreement with our theories.

\begin{figure} [!htbp]
	\centering
	\includegraphics[scale=0.55]{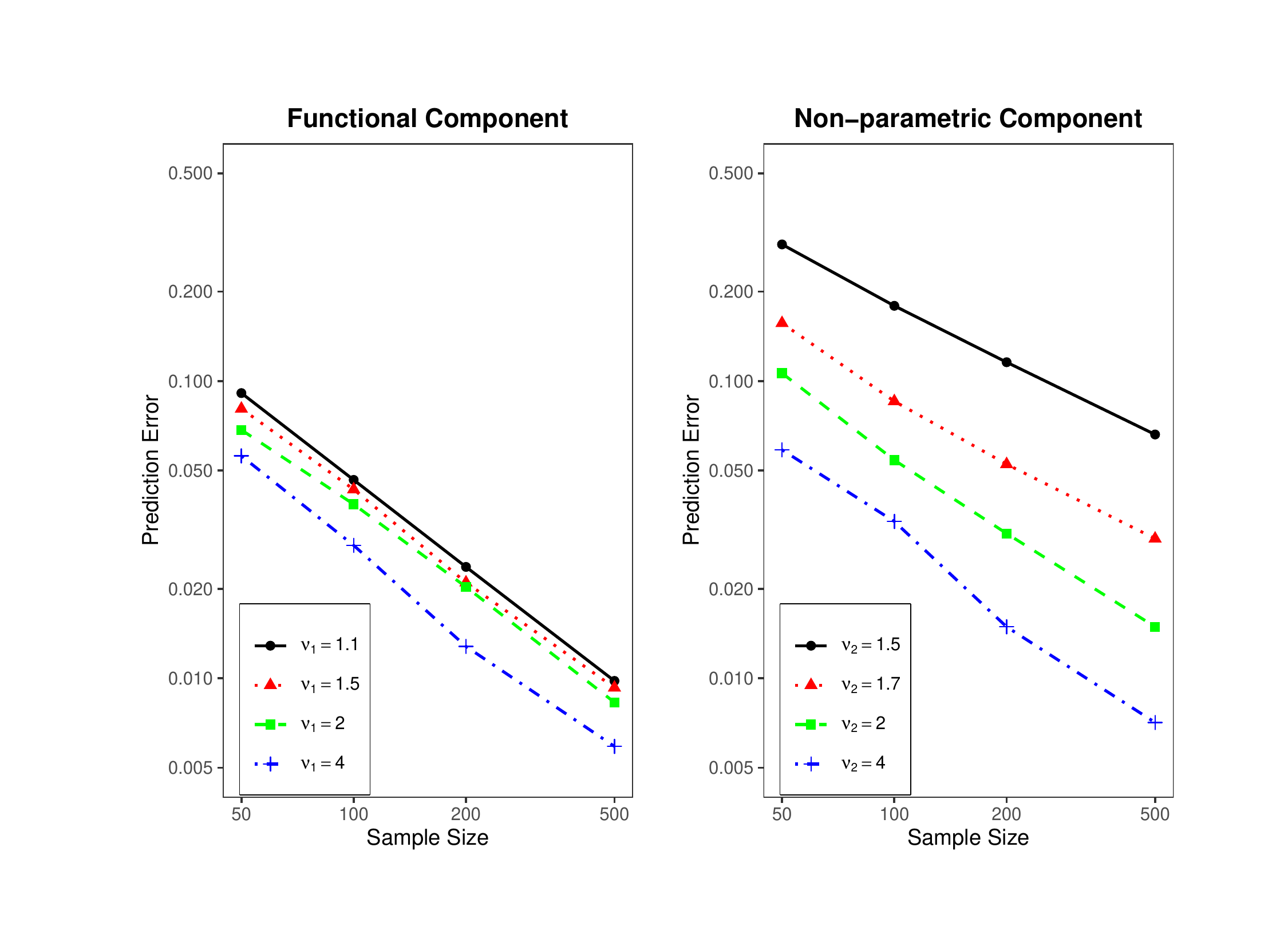}
	\caption{\textit{Left panel: Prediction errors of $\hat{\beta}^{oracle}$ with fixed $\upsilon_2=1.5$. The black solid lines, red dotted lines, green dashed lines and blue dot-dashed lines correspond to the scenarios where $\upsilon_1=1.1,\;1.5,\;2$ and $4$ respectively. Right panel: Prediction errors of $\hat{g}^{oracle}$ with fixed $\upsilon_1=1.1$. The black solid lines, red dotted lines, green dashed lines and blue dot-dashed lines correspond to the scenarios where $\upsilon_2=1.5,\;1.7,\;2$ and $4$ respectively. The performance measure is averaged over 100 simulations. Both axes are in log scale.}}
	\label{Figure3}
\end{figure}


In summary, our simulation results in Section \ref{bothunknown} suggest that the convergence rate for the functional component is solely determined by the smoothness of the functional component and sample size $n$. The rate gets faster as the functional component gets smoother. Moreover, the convergence rate for the nonparametric component is completely determined by the smoothness of the nonparametric component and sample size $n$. Smoother $g^0$ leads to a faster rate of convergence. More importantly, results in Section \ref{g0known} indicate that the convergence rate for the functional component when $g^0$ is unknown is of the same order than that when $g^0$ is known. The results in Section \ref{beta0known} suggest that an accurate estimate of $\beta^0$ does not improve the estimation of $g^0$. The simulation results are in accordance with our theoretical findings in Section \ref{results} where we show the prediction error of the nonparametric component declines at the rate of $\mathcal{O}_p(n^{-2m/(2m+1)})$; whereas the prediction error of the functional component has a rate of convergence $\mathcal{O}_p(n^{-2r/(2r+1)})$. One can also apply our Algorithm \ref{algorithm1} to the case where the scalar variable is multi-dimensional. There are several available kernels including the commonly used Gaussian kernel $G(\bm{z},\bm{w})=\exp(-\lVert\bm{z}-\bm{w}\rVert^2/\sigma^2)$ with tuning parameter $\sigma$ and the polynomial kernel $G(\bm{z},\bm{w})=(\bm{z}^T\bm{w}+1)^d$ with tuning parameter $d$.

\section{Conclusion} \label{conclusion}
In this paper, we propose the double-penalized least squares method for the semi-functional partial linear model within the framework of reproducing kernel Hilbert space (RKHS). At the same time, we introduce an efficient algorithm that requires no iterations to solve for the proposed estimators. Furthermore, under regularity conditions, we establish the minimax optimal rates of convergence for prediction in the semi-functional linear model. We also prove that, with an appropriate choice of regularization parameters, the rate of convergence for the nonparametric component is solely determined by the sample size and the smoothness of the nonparametric component. Additionally, the sample size and the smoothness of the functional component completely determine the rate of convergence for the functional component. Most importantly, for each component, the rate of convergence can be obtained as if the other component were known. The simulation studies further demonstrate the effectiveness of proposed estimators and confirm our theoretical results. 

\section*{Acknowledgments}
The work by J. Fan is partially supported by the Research Grants Council of Hong Kong [Project No. HKBU 12303220] and National Natural Science Foundation of China [Project No. 11801478]. The work by L. Zhu is partially supported by the Research Grants Council of Hong Kong [Project Nos. HKBU12303419 and HKBU12302720] and National Natural Science Foundation of China [Project No. 11671042].

\end{document}